\documentclass[12pt]{amsart}
\usepackage{amssymb}
\usepackage{graphicx}
\usepackage{subfigure}

\newtheorem{theorem}{Theorem}[section]

\newtheorem{lemma}{Lemma}[section]
\newtheorem{cor}{Corollary}[section]

\newcommand{\G}{{\mathcal G}}
\newcommand{\mH}{{\mathcal H}}
\newcommand{\Z}{\mathbb{Z}}
\newcommand{\Q}{\mathbb{Q}}

\title[Graphs as difference graphs of $S$-units]{Representation of finite graphs as difference graphs of $S$-units, I}

\author{K. Gy\H{o}ry}
\address{Institute of Mathematics, University of Debrecen \newline\indent P. O. Box 12, H-4010 Debrecen, Hungary}
\email{gyory@science.unideb.hu}

\author[L. Hajdu]{L. Hajdu$^*$}
\address{Institute of Mathematics, University of Debrecen \newline\indent P. O. Box 12, H-4010 Debrecen, Hungary}
\email{hajdul@science.unideb.hu}

\author{R. Tijdeman}
\address{Mathematical Institute, Leiden University\newline
\indent Postbus 9512, 2300 RA Leiden, The Netherlands}
\email{tijdeman@math.leidenuniv.nl}

\thanks{\noindent $^*$Corresponding author, e-mail: hajdul@science.unideb.hu, tel.:+36-52-512900/22800, fax: +36-52-512728.
\hfill\break
Research was supported in part by the OTKA grants K100339, NK101680. The publication was supported by the T\'AMOP-4.2.2.C-11/1/KONV-2012-0001 project. The project has been supported by the European Union, co-financed by the European Social Fund.}

\subjclass[2010]{05C25, 05C62, 11D61}

\keywords{Arithmetic graphs, cubical graphs, representability, $S$-unit equations}

\begin{document}

\maketitle

\rightline{\sl To Professor I. Z. Ruzsa on his 60th birthday}

\begin{abstract} Let $S$ be a finite non-empty set of primes, $\Z_S$ the ring of those rationals whose denominators are not divisible by primes outside $S$, and $\Z_S^*$ the multiplicative group of invertible elements ($S$-units) in $\Z_S$. For a non-empty subset $A$ of $\Z_S$, denote by $\G_S(A)$ the graph with vertex set $A$ and with an edge between $a$ and $b$ if and only if $a-b\in\Z_S^*$. This type of graphs has been studied by many people.

In the present paper we deal with the representability of finite (simple) graphs $G$ as $\G_S(A)$. If $A'=uA+a$ for some $u\in \Z_S^*$ and $a\in \Z_S$, then $A$ and $A'$ are called $S$-equivalent, since $\G_S(A)$ and $\G_S(A')$ are isomorphic. We say that a finite graph $G$ is {\sl representable / infinitely representable} with $S$ if $G$ is isomorphic to $\G_S(A)$ for some $A$ / for infinitely many non-$S$-equivalent $A$.

We prove among other things that for any finite graph $G$ there exist infinitely many finite sets $S$ of primes such that $G$ can be represented with $S$. We deal with the infinite representability of finite graphs, in particular cycles and complete bipartite graphs. Further, we consider the triangles in $G$ for a deeper analysis. Finally, we prove that $G$ is representable with every $S$ if and only if $G$ is cubical.

Besides combinatorial and numbertheoretical arguments, some deep Diophantine results concerning $S$-unit equations are used in our proofs.

In Part II, we shall investigate these and similar problems over more general domains.
\end{abstract}

\section{Introduction}

Let $R$ be a commutative ring with $1$, and $R^*$ the multiplicative group of units (invertible elements) in $R$. For a non-empty subset $A$ of $R$, denote by $\G(A,R^*)$ the graph with vertex set $A$ and with an edge between vertices $a$ and $b$ if and only if $$a-b\in R^*.$$ For $A=R$, such graphs are special Cayley graphs. They were introduced in Gy\H{o}ry \cite{new11a}, \cite{gy1980} for the case when $R$ is the ring of integers of any number field, and in Gy\H{o}ry \cite{gy1982} for any integral domain $R$ of characteristic $0$. In these works it is shown that the connectedness properties of the graphs $\G(A,R^*)$ with finite $A$, and of their complements, play an important role in the investigation of several diophantine problems concerning the irreducibility of certain polynomials, decomposable form equations, polynomials and integral elements of given discriminant and power integral bases. In Gy\H{o}ry \cite{gy1980}, \cite{gy1982}, \cite{gy1990}, \cite{gy1992}, \cite{gy2008} the structure of these graphs was described from the point of view of connectedness. For related results and applications, we refer to Gy\H{o}ry \cite{new11b}, \cite{new13a}, \cite{new13b}, \cite{new17a}, Evertse, Gy\H{o}ry, Stewart and Tijdeman \cite{new7a}, Leutbecher \cite{new21a}, Leutbecher and Niklash \cite{new21b}, Gy\H{o}ry, Hajdu and Tijdeman \cite{new18a}, Ruzsa \cite{ruzs} and the references there.

Independently of these investigations, many people considered the graph $\G(R,R^*)$ in the case when $R$ is a finite ring or an Artinian ring. Various properties of these graphs, including connectivity, diameters and chromatic numbers, were studied among others in the papers by Dejter and Giudici \cite{newX}, Berrizbeitia and Giudici \cite{new25}, Fuchs \cite{new26}, Klotz and Sander \cite{new27}, Lucchini and Mar\'oti \cite{new24}, Lanski and Mar\'oti \cite{lm}, Akhtar, Boggess, Jackson-Henderson, Jim\'enez, Karpman, Kinzel and Pritkin \cite{abjkp} and Khashyarmanesh and Khorsandi \cite{new30}. In these works the graph $\G(R,R^*)$ is usually called unitary Cayley graph.

Erd\H{o}s and Evans \cite{new31} showed that every finite graph is isomorphic to $\G(A,(\Z/n\Z)^*)$ for some positive integer $n$ and $A\subseteq\Z/n\Z$, where $\Z/n\Z$ denotes the integers modulo $n$. In other words, any finite graph is representable as $\G(A,(\Z/n\Z)^*)$ for an appropriate $n$ and $A\subseteq\Z/n\Z$.

In the present paper we continue the investigations of Gy\H{o}ry and Ruzsa, and deal with the representability of finite graphs $G$ as $\G_S(A):=\G(A,\Z_S^*)$, where $S$ is a finite set of primes, $\Z_S$ the ring of $S$-integers in $\mathbb{Q}$, i.e. the set of those rationals whose denominators are not divisible by primes outside $S$, $\Z_S^*$ the group of $S$-units and $A$ a finite subset of $\Z_S$. If $A'=uA+a$ for some $u\in \Z_S^*$ and $a\in \Z_S$, then $A$ and $A'$ are called $S$-equivalent, since $\G_S(A)$ and $\G_S(A')$ are isomorphic. We say that a finite graph $G$ is {\sl representable / infinitely representable} with $S$ if $G$ is isomorphic to $\G_S(A)$ for some $A$ / for infinitely many non-$S$-equivalent $A$.

We present new results in Sections \ref{s2}-\ref{s-i}.
In Section \ref{s2} it is proved that for any finite graph $G$ there exist infinitely many finite sets $S$ of primes such that $G$ is representable with $S$. We deal with the infinite representability of finite graphs, in Section \ref{s-b} of cycles and complete bipartite graphs, and in Section \ref{s3} of more general graphs. Subsequently in Section \ref{s-tri} the triangles in $G$ are used for a deeper analysis. We show that if the complement of $G$ has either at least three components, or two components of order $\geq 2$, and if the order of $G$ is greater than $3\cdot 2^{16(|S|+1)}$, then $G$ is not representable with $S$. Here $|S|$ denotes the cardinality of $S$. In Section \ref{s-i} we state that $G$ is representable with every $S$ if and only if $G$ is cubical, i.e. embeddable in $\{0,1\}^n$ for some $n$. The Sections \ref{s-p1} to \ref{s-p5} contain the proofs of the statements in Sections \ref{s2} to \ref{s-i}, respectively.

In the proofs combinatorial and numbertheoretical arguments are combined with some deep results on $S$-unit equations and on the graphs $\G_S(A)$, which were established by the Thue-Siegel-Roth-Schmidt method from Diophantine approximation.

\section{Representability of graphs}
\label{s2}

Let $S$ be a finite set of primes, $\Z_S$ the ring of $S$-integers in $\Q$ and $\Z_S^*$ the group of $S$-units.

For any ordered subset $A$ of $\Z_S$, we denote by $\G_S(A)$ the graph whose vertices are the elements of $A$ and whose edges are the (unordered) pairs $\{a_i,a_j\}$ with $a_i,a_j\in A$ for which
$$
a_i-a_j\in \Z_S^*;
$$
cf. Gy\H{o}ry \cite{gy1980} where the complements of these graphs were studied. The ordered subsets $A$ and $A'$ of $\Z_S$ are called $S$-equivalent if
$$
A'=uA+b
$$
for some $u\in \Z_S^*$ and $b\in \Z_S$. As we observed before, the graphs $\G_S(A)$ and $\G_S(A')$ are then isomorphic.

Throughout the paper, all graphs we consider are simple. By the order of a graph $G$ we mean the number of its vertices, denoted by $|G|$. By a component of $G$ we mean a connected component. We recall that a graph $G$ is {\sl representable with} $S$ if there is a subset $A$ of $\Z_S$ such that $\G_S(A)$ is isomorphic to $G$. Similarly, $G$ is said to be {\sl finitely representable with} $S$ if $G$ is isomorphic to some $\G_S(A)$, but only to finitely many of them, up to $S$-equivalence. Further, $G$ is said to be {\sl infinitely representable with} $S$ if $G$ is isomorphic to $\G_S(A)$ for infinitely many pairwise nonequivalent $A$.

In this section we formulate some basic results.

\begin{theorem} \label{thm1}
For any finite graph $G$ there exist infinitely many finite sets $S$ of primes such that $G$ is representable with $S$.
\end{theorem}

As usual, by a forest graph we mean a graph containing no cycles, i.e. a finite, disjoint union of trees.

\begin{theorem}\label{thm3}
Let $S$ be any fixed finite set of primes, and $G$ be a finite forest graph. Then $G$ can be represented with $S$.
\end{theorem}

In fact, Theorem \ref{thm3} is a simple consequence of the following result.

\begin{theorem} \label{thm4}
Let $S$ be any fixed finite set of primes, and $A$ be any fixed finite set of $S$-integers.

\vskip.2truecm

\noindent
{\rm i)} There exist infinitely many  $a'\in \Z_S$ outside $A$ such that writing $A'=A\cup \{a'\}$, $a'$ is an isolated vertex of $\G_S(A')$.

\vskip.2truecm

\noindent
{\rm ii)} For every $a\in A$ there exist infinitely many $a'\in \Z_S$ such that writing $A'=A\cup \{a'\}$, in $\G_S(A')$ the vertex $a'$ is connected by an edge with $a$ only.
\end{theorem}

The following result shows that the investigations can be reduced to components of a graph.

\begin{theorem}\label{thmuj}
Let $S$ be any fixed finite set of primes, and suppose that every component of a graph $G$ can be represented with $S$. Then $G$ can be represented with $S$.
\end{theorem}

\section{Cyclic and bipartite graphs}
\label{s-b}

Let $S$ be a finite non-empty set of primes, $\Z_S$ the set of $S$-integers and $\Z_S^*$ the group of $S$-units in $\mathbb{Q}$. Given a cycle
$$a_1 \to a_2 \to \cdots \to a_n \to a_1$$
in $\G_S(\Z_S)$, the 'labels' of the edges, $u_i = a_{i+1} - a_i$ for $i=1, \dots, n-1$ and $u_n=a_1 - a_n$, satisfy
$$u_1 + u_2 + \cdots +u_n = 0.$$
We call the cycle nondegenerate if there is no non-empty proper zero subsum
$$u_{i_1} + u_{i_2} + \dots + u_{i_m} = 0, ~~1 \leq i_1 < \cdots < i_m \leq n,~0 < m < n.$$
Ruzsa \cite{ruzs} proved the following results.
\vskip 2mm
\noindent
i) If $2\in S$, then there are nondegenerate cycles of every length among the induced subgraphs of $\G_S(\Z_S)$. \\
ii) If $2\notin S$, then there are cycles of every even length among the induced subgraphs of $\G_S(\Z_S)$ and none of odd length. \\
iii) If $2\notin S$ and $3 \in S$, then there are nondegenerate cycles of every even length among the induced subgraphs of $\G_S(\Z_S)$.
\vskip 2mm
\noindent
Ruzsa conjectured that, if $2 \notin S$, then there are nondegenerate induced cycles of every sufficiently large even length.
He further proved that, for any $\varepsilon > 0$, any subgraph of $\G_S(\Z_S)$ on $n$ vertices has average degree $< c_{\varepsilon, S} n^{\varepsilon}$.

We say that a graph $G$ is doubly connected if after deleting any edge of $G$, the graph obtained is connected.
If $G$ is not doubly connected, we say that it is at most simply connected. We consider some doubly connected graphs. Let $C_n$ denote the cyclic graph of order $n$, and $K_{m,n}$ the complete bipartite graph of type $(m,n)$.

\begin{theorem}
\label{prop3.6}
\noindent
{\rm i)} The graphs $C_{2n}$ $(n\geq 2)$ and $K_{2,2}$ are infinitely representable with all $S$. \\
{\rm ii)}  The graphs $C_3$, $C_5$ and $K_{m,n}$ with $m>n>1$ or $m = n \geq 3$ are finitely representable with every $S$.
\end{theorem}

\noindent It depends on $S$ whether $C_{2n+1}$ for $n>2$ is infinitely representable.

A large complete bipartite graph is not representable with $S$.

\begin{theorem}
\label{thm3.7}
If $m>1$, $n>1$ and
\begin{equation}
\label{star}
m+n>3\cdot 2^{16(|S|+2)}
\end{equation}
then $K_{m,n}$ is not representable with $S$.
\end{theorem}

\section{Some results on infinite representability}
\label{s3}

From now on, a graph $G$ will mean a finite simple graph.

We present two theorems which show that under suitable circumstances representability implies infinite representability. Our next result shows that the representability of a graph $G$ with a special $S$ over ${\mathbb Z}$ is already sufficient for the infinite representability of $G$ with all $S$.

\begin{theorem}
\label{thm2}
Suppose that a graph $G$ with $|G|\geq 3$ is representable with some $S_0$ of the form $S_0=\{p\}$, where $p$ is a prime larger than twice the number of edges of $G$. Then $G$ is infinitely representable with all finite sets $S$ of primes.
\end{theorem}

Now we provide two simple consequences of the above result.

\begin{cor} \label{cor4.1}
Let $S$ be a finite set of primes. Let $G$ be a graph which is finitely representable with $S$. Then there exist infinitely many sets of primes $S'$ such that $G$ is not representable with $S'$.
\end{cor}

\begin{cor} \label{cor4.2}
Let $G$ with $|G|\geq 3$ be representable with every $S$.
Then $G$ is infinitely representable with every $S$.
\end{cor}

Our last theorem in this section shows that certain graphs are such that for any $S$, they are either not representable with $S$ or they are infinitely representable with $S$.

\begin{theorem}
\label{thm3.1}
Let $G$ be a graph with $|G|\geq 3$ which is at most simply connected. If $G$ is representable with some $S$, then it is infinitely representable with $S$.
\end{theorem}

\section{$\triangle$-connectedness} \label{s-tri}

For a graph $G$ we denote by $G^\triangle$ the triangle graph (or $\triangle$-graph) of $G$, i.e. the graph whose vertices are the edges of $G$, and two vertices $e_1$ and $e_2$ of $G^\triangle$ are connected by an edge if and only if $G$ contains a triangle having $e_1$ and $e_2$ as edges. Further, if $G$ and $G^\triangle$ are connected then we say that $G$ is $\triangle$-connected. Figure 1 shows some examples. The $\triangle$-graph of tree and forest graphs have only isolated vertices. The third graph of Figure 1 is doubly connected, but not $\triangle$-connected.

\begin{figure}[htp!]
\label{fig1}
\begin{center}
   \includegraphics[scale=0.6]{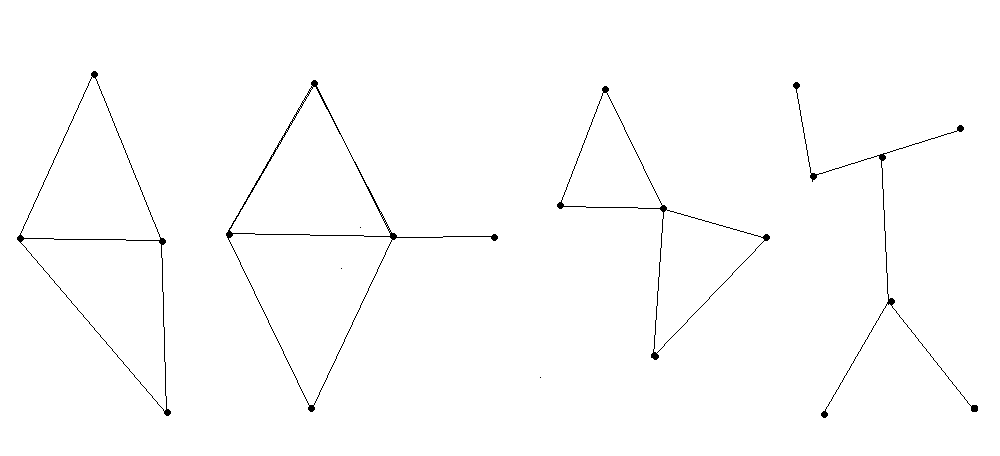}
\caption{The first graph is $\triangle$-connected, while the last three ones are not $\triangle$-connected.}
\end{center}
\end{figure}

We note that if a graph $\G_S(A)$ contains a triangle then there are exceptional $S$-units in $\Z_S$, i.e. $S$-units $u$ such that $1-u$ is also $S$-unit. Such units do not always exist. E.g. if the primes in $S$ are all odd, then it is easy to see that $\Z_S$ has no exceptional $S$-units. Consequently, the corresponding graphs $\G_S(A)$ cannot have triangles. On the other hand, we recall that for a given graph $G$ there are infinitely many pairs $(S,A)$ for which $\G_S(A)$ is isomorphic to $G$.

The following theorem is a partial counterpart of Theorem \ref{thm3.1}. Observe that if both $G$ and $G^\triangle$ are connected, then $G$ is doubly connected.

\begin{theorem}
\label{thm3.3}
Let $G$ be a graph of order $\geq 3$ such that both $G$ and $G^\triangle$ are connected. Then $G$ is finitely representable with every $S$.
\end{theorem}

Note that cyclic graphs $C_n$ with $n>3$ and bipartite graphs $K_{m,n}$ with $m>1,n>1$ are doubly connected, but not $\triangle$-connected. According to Theorem \ref{prop3.6} $C_{2n}$ $(n\geq 2)$ and $K_{2,2}$ are infinitely representable with every $S$, but $C_3$, $C_5$ and $K_{m,n}$ for $m>n>1$ or $m=n\geq 3$ are finitely representable with every $S$. Thus some doubly connected graphs which are not $\triangle$-connected are infinitely representable with every $S$ and some others are finitely representable with every $S$.

Theorem \ref{thm3.3} can be generalized in the following way.
We denote by $\mH(G)$ the graph whose vertices are the $\triangle$-connected components of $G$, and two vertices of $\mH(G)$ are connected if the corresponding $\triangle$-connected components of $G$ have at least two vertices in common in $G$. This graph $\mH(G)$ will be called the $\mH(G)$-graph of $G$. Figure 2 shows an example.

\begin{figure}[htp!]
\label{fig2}
\begin{center}
   \includegraphics[scale=0.6]{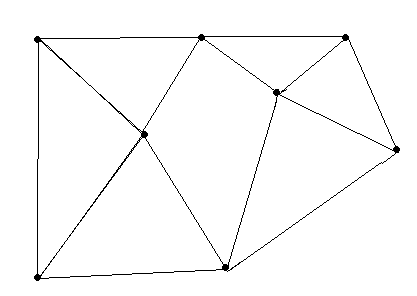}
\caption{A graph $G$ whose $\triangle$-hypergraph is not connected, but for which ${\mH}(G)$ is connected. In this case ${\mH}(G)$ consists of two vertices connected by an edge.}
\end{center}
\end{figure}

\begin{theorem}
\label{thm3.4}
Let $G$ be a graph of order $\geq 3$. Suppose that both $G$ and $\mH(G)$ are connected. Then $G$ is finitely representable with every $S$.
\end{theorem}

\noindent If $G^\triangle$ is connected, then $\mH(G)$ consists of one vertex and is therefore also connected. Hence Theorem \ref{thm3.3} is a special case of Theorem \ref{thm3.4}.

The third graph of Figure 1 is an example of a doubly connected graph $G$ which is infinitely representable with some $S$. Here $\mH(G)$ consists of two isolated vertices. By Theorem \ref{thm3.4} the graph of Figure 2 is finitely representable with all $S$.

Theorems \ref{thm3.3} and \ref{thm3.4} are applications of the following theorem. We denote the complement of $G$ by $\overline{G}$.

\begin{theorem}
\label{thm5.4}
Let $k\geq 3$ be an integer, and fix $S$. Then for all but at most
$$
\left(k\cdot 5^{162(3|S|+4)}\right)^{4(k-1)}
$$
$S$-equivalence classes of ordered $k$-term subsets $A$ from $\Z_S$, one of the following cases holds:

\vskip.2truecm

\noindent
{\rm i)} $\overline{\G_S(A)}$ is connected and at least one of $\G_S(A)$ and $\G_S(A)^\triangle$ is not connected;

\vskip.2truecm

\noindent
{\rm ii)} $\overline{\G_S(A)}$ has exactly two components, $\overline{\G_1}$ and $\overline{\G_2}$, say, such that $|\overline{\G_1}|=1$, and $\G_2$ is not connected;

\vskip.2truecm

\noindent
{\rm iii)} $k=4$ and $\G_S(A)=K_{2,2}$.
\end{theorem}

\noindent As is pointed out in \cite{gy1980}, Section \ref{s-b}, each of the cases i), ii), iii) may occur. Moreover, for each of i), ii), iii), one can choose $S$ such that there are infinitely many $S$-equivalence classes of ordered $k$-term subsets $A$ in $\Z_S$ with the above property.

The following consequence of Theorem \ref{thm5.4} is a quantitative refinement of Theorem \ref{thm3.3}.

\begin{theorem}
\label{thm3.8}
Let $G$ be a graph of order $k\geq 3$ and $S$ a finite set of prime numbers. Suppose that $G$ is isomorphic to $G_S(A)$ for more than
$$
\left(k\cdot 5^{162(3|S|+4)}\right)^{4(k-1)}
$$
$S$-equivalence classes of ordered subsets $A$ from $\Z_S$.  Then at least one of $G$ and $G^\triangle$ is not connected.
\end{theorem}

\vskip.2truecm

\noindent{\bf Question 1.} Does there exist a criterion/algorithm to decide the infinite representability of a graph $G$ for fixed $S$?

\vskip.2truecm

Finally, the following result is concerned with the situation where no representation is possible.

\begin{theorem}
\label{thm3.6}
Let $G$ be a graph of order $k$ such that $\overline{G}$ has either at least three components, or two components of order $\geq 2$. If
$$
k>3\cdot 2^{16(|S|+1)}
$$
then $G$ is not representable with any $S$.
\end{theorem}

\vskip.2truecm

\noindent{\bf Question 2.} Does there exist a criterion/algorithm to decide the representability of a graph $G$ for any given $S$?

\section{Graphs which are representable with all $S$} \label{s-i}

Theorem \ref{prop3.6} states that $G=C_{2n}$ for $n\geq 2$ and $G=K_{2,2}$ are representable with all $S$. We want to study such graphs. It follows from Corollary \ref{cor4.2} that if $G$ with $|G|\geq 3$ is representable with all $S$, then it is infinitely representable with all $S$.

The $n$-cube $Q_n$ is defined as the graph of which the vertices are $n$-tuples with coordinates 0 and 1 and in which two vertices are connected by an edge if and only if the vertices differ in exactly one coordinate. Hence $Q_n$ has $2^n$ vertices and $n2^{n-1}$ edges. An embedding of a graph $G$ into $Q_n$ is an injective mapping of the vertices of $G$ into the vertices of $Q_n$ which maps the edges of $G$ into edges of $Q_n$. A graph which can be embedded in $Q_n$ for some $n$ is called cubical. Obviously, a cubical graph is bipartite. The converse is not true; the graph $K_{2,3}$ is an example of a bipartite graph which is not cubical. All trees are cubical \cite{firs}.

Several authors have published results on cubical graphs. We cite the ones which are the most relevant for us. For more details we refer to the survey paper  \cite{hhw}.

Garey and Graham \cite{gg} call a graph $G$ critical if it is not cubical and every proper subgraph $H$ of $G$ is cubical. It is clear that any odd cycle $C_{2n+1}$ is critical. The smallest bipartite graph which is critical is the bipartite graph $K_{2,3}$. They show that the number of non-isomorphic critical graphs on $n$ vertices is exponential in $n$. Garey and Graham as well as Gorbatov and Kazanskiy \cite{gk} have given procedures for constructing critical graphs from smaller critical graphs.

Havel and Moravek \cite{hm} found a criterion for a graph $G$ to be cubical based on so-called $c$-valuations. A $c$-valuation of a bipartite graph $G$ is a labeling of the edges of $G$ such that
\begin{itemize}
\item for each cycle in $G$, all distinct edge labels occur an even number of times;
\item for each (noncyclic) path in $G$, there exists at least one edge label which occurs an odd number of times.
\end{itemize}
\noindent
The dimension of a $c$-valuation is the number of edge labels used. It is shown in \cite{hm} that a graph $G$ is cubical with $G \subseteq Q_n$ if and only if there exists a $c$-valuation of $G$ of dimension $n$.
Intuitively, the labels of the edges correspond with the directions of the edges in an $n$-cube embedding of $G$.

Afrati {\it et al.} \cite{app} have shown that telling whether a graph is cubical is NP-complete.

We shall prove the following equivalence.

\begin{theorem} \label{thmcub}
A graph $G$ is representable with all $S$ if and only if $G$ is cubical.
\end{theorem}

Note that since forest graphs are cubical, the above theorem immediately implies Theorem \ref{thm3}. Further, the above result together with Theorem \ref{thm2} implies that a graph is cubical if and only if it can be represented with $S_0$ specified in Theorem \ref{thm2}.

\section{Proofs of the results stated in Section \ref{s2}}
\label{s-p1}

In the proofs below we shall work with finite subsets $A$ of $\Z$. Every $S$-equivalence class of ordered subsets $A$ from $\Z_S$ contains a subset consisting of integers. Such a subset can be obtained from $A$ by multiplying it by an appropriate element of $Z_S^*\cap \Z$. Hence for Theorems \ref{thm1}-\ref{thm4} it suffices to study the graphs $\G_S(A)$ with subsets $A$ having all the elements from $\Z$. In this case, $a,b\in A$ are connected by an edge if and only if $a-b\in \Z_S^*\cap \Z$.

\begin{proof}[Proof of Theorem \ref{thm1}]
Let $G$ be a fixed graph with $|G|=n$. Write $n':=\max\{n,3\}$ and
$$
S_0:=\{p\ \text{prime}\ :\ p<n'\}.
$$
We prove by induction on $k$ that for any graph $G_k$ with
$|G_k|=k\leq n$ there exists a finite set $S_k$ of primes
with $S_0 \subseteq S_k$  and a finite set $A_k\subseteq \Z$ with $|A_k|=k$ such that $\G_{S_k}(A_k)$ is isomorphic to $G_k$.

Let $k=1$. Then $G_1$ is a graph with one vertex (and without
edges). Taking any finite set of primes $S_1$ with $S_1\supseteq S_0$ and $A_1=\{0\}$, we are obviously done in this case.

Let now $G_k$ be a graph such that $|G_k|=k$ with $2\leq
k\leq n$. Write $G_k=\{v_1,\dots,v_k\}$, and $G_{k-1}=G_k\setminus
\{v_k\}$ (removing also the corresponding edges). By induction
we may assume that there exists a set $S_{k-1}$ of primes including $S_0$ as a subset and a set
$A_{k-1}=\{a_1,\dots,a_{k-1}\}$ of integers such that
$\G_{S_{k-1}}(A_{k-1})$ is isomorphic to $G_{k-1}$, by an isomorphism $\varphi:\ \G_{S_{k-1}}(A_{k-1})\to G_{k-1}$. Without loss of generality we may assume that $\varphi(a_i)=v_i$ $(i=1,\dots,k-1)$. Write $T''$ for the set of indices of those vertices of $G_{k-1}$ which are {\bf not} connected with $v_k$ by an edge in $G_{k}$. Further, put
$$
D:=\{d\ \text{prime}\ :\ d\notin S_{k-1},\ d\mid a-b\
\text{for some distinct}\ a,b\in A_{k-1}\}.
$$
For later use, observe that for all $d\in D$ we have $d\geq n'>k-1$.

If $T''\neq\emptyset$, write $T''=\{t_1,\dots,t_\ell\}$, and choose distinct primes $q_{t_1},\dots,q_{t_\ell}$ such that for all $t_j\in T''$ we have $q_{t_j}\notin S_{k-1}\cup D$.
Observe that by these choices, for any distinct
$i_1,i_2\in\{1,\dots,k-1\}$ we have $a_{i_1}\not\equiv
a_{i_2}\pmod{q_{t_j}}$. For each prime $d\in D$ choose an $x_d\in \Z$ such that for all $i=1,\dots,k-1$ we have
\begin{equation} \label{nr1}
a_i\not\equiv x_d\pmod{d}.
\end{equation}
Since $d>k-1$ for all $d\in D$, such $x_d$ exist. Consider now the following linear system of congruences:
\begin{equation} \label{nr2}
\begin{cases}
a\equiv x_d\pmod{d} & (d\in D),\\
a\equiv a_{t_j}\pmod{q_{t_j}} & (t_j\in T'').
\end{cases}
\end{equation}
If $T''=\emptyset$ then the second set of congruences is empty. By the Chinese Remainder Theorem, this system has infinitely many solutions $a$. Choose $a_k$ to be an arbitrary solution, and let $A_k=A_{k-1}\cup\{a_k\}$. Further, put $T'=\{1,\dots,k-1\}\setminus T''$ and set
$$
S_k=S_{k-1}\cup\{p\ \text{prime}\ :\ p\mid a_k-a_i\ \text{for some}\ i\in T'\}.
$$

We claim that by these choices the graph $\G_{S_k}(A_k)$ is
isomorphic to $G_k$. More precisely, an isomorphism is given by $\varphi^*:\ \G_{S_k}(A_k)\to G_k$ with $\varphi^*(a_i)=v_i$
$(i=1,\dots,k)$.

Let $i\in\{1,\dots,k-1\}$. If $i\in T'$ then on the one hand,
$v_i$ and $v_k$ are connected by an edge in $G_k$, and on the other hand, by the definition of $S_k$ we have that $a_i$ and $a_k$ are connected in $\G_{S_k}(A_k)$. Assume now that $i\in T''$. Then $v_i$ and $v_k$ are not connected in $G_k$. Writing $i=t_j$, in view of $q_{t_j}\notin S_{k-1}$ and $q_{t_j}\mid a_k-a_i$, we have that $q_{t_j}\notin S_k$. Indeed, otherwise $q_{t_j}\mid a_k-a_{i'}$ for some $i'\in T'$, whence $q_{t_j}\mid a_i-a_{i'}$ with distinct $i,i'\in\{1,\dots,k-1\}$. This means that $q_{t_j}\in S_{k-1}\cup D$, which contradicts its definition. Thus $q_{t_j}\mid a_k-a_i$ implies that $a_i$ and $a_k$ are not connected by an edge in $\G_{S_k}(A_k)$.

Finally, we need to check that for any $i,j\in\{1,\dots,k-1\}$, $a_i$ and $a_j$ are connected by an edge in $\G_{S_k}(A_k)$ if and only if they are connected by an edge in $\G_{S_{k-1}}(A_{k-1})$. If $a_i$ and $a_j$ are connected by an edge in $\G_{S_{k-1}}(A_{k-1})$ then by $S_{k-1}\subseteq S_k$, obviously they are connected by an edge in $\G_{S_k}(A_k)$. Assume now that $a_i$ and $a_j$ are not connected in $\G_{S_{k-1}}(A_{k-1})$. Then there is a prime $d\in D$ dividing $a_i-a_j$. Observe that, by \eqref{nr2} and \eqref{nr1}, $d\mid a_k-x_d$ and $d\nmid a_\ell-x_d$, whence $d\nmid a_k-a_\ell$ for $\ell=1,\dots,k-1$. This implies that $d\notin S_k$. Hence $a_i$ and $a_j$ are not connected by an edge in $\G_{S_k}(A_k)$ either.

The above argument by induction shows the existence of a set $S=S_n$ with the required property. The infinitude of such sets $S$ can be guaranteed in the following way. If $G$ has no edges, then the statement is obvious. Otherwise, we may assume that the degree of $v_n$ is positive. Choose an arbitrary prime $p$ outside $S_n$, which is different from all the primes appearing as a modulus in \eqref{nr2} on constructing $a_n$. Observe that after extending \eqref{nr2} with the congruence
$$
a\equiv a_i \pmod{p}
$$
with some $i\in T'$ in the $n$-th step, the new system is also solvable. Taking a solution $a_n^*$ of this system in place of $a_n$, we see that $p\in S=S_n^*\neq S_n$ for the set $S$ obtained in this way. Now we may choose another prime outside $S_n\cup S_n^*$, etc., and the theorem follows.
\end{proof}

\begin{proof}[Proof of Theorem \ref{thm4}] Write $A=\{a_1,\dots,a_n\}$.

To prove i) choose primes $q_1,\dots,q_n$ outside $S$. Consider the system of linear congruences
$$
x\equiv a_i\pmod{q_i}\ \ \ (i=1,\dots,n)
$$
in $x\in \Z$. By the Chinese Remainder Theorem, this system has infinitely many solutions. Let $a'\in \Z$ be a solution such that $a'\notin A$. Then obviously, $a'$ is an isolated vertex of the graph $\G_S(A')$ where $A'=A\cup \{a'\}$.

To prove ii), take an arbitrary $a\in A$. Write
$$
D:=\{\pm (a_i-a_j)\ :\ 1\leq i<j\leq n\},
$$
and let $u\in \Z_S^*\cap \Z$ such that $u\notin D$ and for any $w\in \Z_S^*\cap \Z$ we also have $u+w\notin D$. The existence of such a $u$ easily follows from the theory of $S$-unit equations. Namely, for any $d\in D$ the equation $u+v=d$ has only finitely many solutions in $u,v\in \Z_S^*\cap \Z$, see \cite{gy1979} or Theorem A. Avoiding all such elements $u,v$, in fact we can choose $u$ in infinitely many ways. Let $a'=a+u$. Then $a'\notin A$, and obviously $a'$ and $a$ are connected by an edge in the graph $\G_S(A')$ where $A'=A\cup \{a'\}$. Assume that $a'$ is also connected with some vertex $b\in A$ with $b\neq a$. Then $b-(a+u)=w\in \Z_S^*\cap \Z$. However, this yields $w+u=b-a$, whence $w+u\in D$, contradicting the choice of $u$. This shows that in the graph $\G_S(A')$ only the vertex $a'$ is connected by an edge with the vertex $a$.
\end{proof}

\begin{proof}[Proof of Theorem \ref{thm3}]
Let $G$ be the disjoint union of the tree graphs $T_1,\dots,T_k$. Starting from one vertex $a\in \Z$, using part ii) of Theorem \ref{thm4}, we can gradually build up a set $A_1\subseteq \Z$ such that $\G_S(A_1)$ is isomorphic to $T_1$. Then by part i) of Theorem \ref{thm4} we can adjoin an isolated vertex $a'\in \Z$ to this graph, and then build up a component $A_2\subseteq \Z$ (with $a'\in A_2$) such that $\G_S(A_2)$ is isomorphic to $T_2$. Following this procedure, we can clearly construct a set $A=A_1\cup A_2\cup \dots \cup A_k$ with the desired property.
\end{proof}

\begin{proof}[Proof of Theorem \ref{thmuj}]
If $G$ is connected, i.e. $G$ has only one component, then the statement is trivial. Suppose that the statement is true for graphs having at most $k$ components with $k\geq 1$, and let $G$ be a graph having $k+1$ components. Let $G'$ be a component of $G$, and put $G''=G\setminus G'$. Let $A'$ and $A''$ be subsets of $\Z$ such that $G'$ and $G''$ are isomorphic to $\G_S(A')$ and $\G_S(A'')$, respectively. Then, similarly as in the proof of Theorem \ref{thm4} ii), we can choose a $u\in \Z$ such that $A'+u$ and $A'$ are disjoint, and $\G_S(A'+u)$ and $\G_S(A'')$ have no vertices connecting these graphs by an edge. Hence the statement follows by induction.
\end{proof}

\section{Proofs of the results stated in Section \ref{s-b}}
\label{s-p2}

In the proof of Theorem \ref{prop3.6}, we shall use the following deep finiteness result. Consider the $S$-unit equation
\begin{equation}
\label{eq5.1}
ax+by=1\ \ \ \text{in}\ x,y\in \Z_S^*,
\end{equation}
where $a,b$ are non-zero elements of $\Q$.
\vskip 3mm

\noindent
{\bf Theorem A.} (Evertse \cite{ev}){\sl
The number of solutions of \eqref{eq5.1} is at most
\begin{equation}
\label{neweq5}
3\cdot 7^{2|S|+3}.
\end{equation}
}

Consider the generalization
\begin{equation}
\label{neweq7}
a_1x_1+\dots+a_nx_n=1\ \ \ \text{in}\ x_1,\dots,x_n\in \Z_S^*
\end{equation}
of equation \eqref{eq5.1}, where $a_1,\dots,a_n$ are non-zero elements of $\Q$. A solution $(x_1,\dots,x_n)$ of \eqref{neweq7} is called {\it non-degenerate} if
$$
\sum\limits_{i\in I} a_ix_i\neq 0\ \ \ \text{for each non-empty subset}\ I\ \text{of}\ \{1,\dots,n\}
$$
and {\it degenerate} otherwise. It is clear that for $n=2$ each solution is non-degenerate. Evertse, Schlickewei and Schmidt \cite{ess} gave an explicit upper bound for the number $N_n$ of non-degenerate solutions of \eqref{neweq7}. This has been improved by Amoroso and Viada \cite{av} to the following result.
\vskip 3mm

\noindent
{\bf Theorem B.} {\sl
\begin{equation}
\label{neweq8}
N_n\leq (8n)^{4n^4(n|S|+n +1)}.
\end{equation}}

\begin{proof}[Proof of Theorem \ref{prop3.6}]
i) Let $n$ be an integer $\geq 2$ and $G=C_{2n}$. Then $C_{2n}$ is infinitely representable with all $S$ according to the proof by induction of Theorem 3.1 of \cite{ruzs}. (The result follows also from Theorem \ref{thmcub}.)
\vskip 2mm

Let $G=K_{2,2}$, and let $S$ be arbitrary. Let $u$ be a fixed $S$-unit. By Theorem A there are infinitely many $S$-units $w$ such that in
$$
u=(u+w)-w=(u-w)+w,
$$
none of $u+w$ and $u-w$ is an $S$-unit. Further, it is easy to see that for such $w$ the ordered subsets $A_w=(0,u,w,u+w)$ are pairwise non-$S$-equivalent, and the graphs $\G_S(A_w)$ are bipartite and so isomorphic to $K_{2,2}$.
\vskip 2mm

\noindent
ii) Let $G=C_3$. Then every representation of $G$ with $S$ corresponds with a normalized equation $x+y=1 \ \text{in}\ x,y\in \Z_S^*$. By Theorem A the number of solutions of this equation is finite. Therefore $C_3$ is finitely representable with $S$.

\vskip 2mm

Let $G=C_5$. Let $A=\{a_1,\dots,a_5\}\subseteq \Z_S$
be such that $\G_S(A)$ is isomorphic to $G$. Without loss of generality we may assume that $$a_2-a_1,\ a_3-a_2,\ a_4-a_3,\ a_5-a_4,\ a_1-a_5\in \Z_S^*.$$ Write $u_1,\dots,u_5$ for these $S$-units, respectively. Then we have
\begin{equation}
\label{ujtrivi}
u_1+\dots+u_5=0.
\end{equation}
Suppose that there is a vanishing subsum in the left hand side of \eqref{ujtrivi}. We may suppose that we have such a subsum consisting of two terms. Since these terms cannot be consecutive ones, without loss of generality we may assume that $u_1+u_3=0$. Then, as one can easily check, we have that $a_4-a_1=a_3-a_2$ is an $S$-unit. However, then $a_1$ and $a_4$ are also connected by an edge in $\G_S(A)$, which means that this graph is not isomorphic to $G$. Hence we get that the left hand side of the equation \eqref{ujtrivi} has no vanishing subsums. By Theorem B the number of non-degenerate solutions of equation \eqref{ujtrivi} is finite. Therefore $C_5$ is finitely representable.

\vskip 2mm

Let $G=K_{m,n}$ with $m>n>1$ or $m=n\geq 3$. Then the assertion immediately follows from Theorem \ref{thm5.4}, since $\overline{K_{m,n}}$, the complement of $K_{m,n}$ has two components each of size $\geq 2$. We remark that Theorem \ref{prop3.6} ii) is not utilized in the proof of Theorem \ref{thm5.4}.
\end{proof}

\begin{proof} [Proof of Theorem \ref{thm3.7}]
Theorem \ref{thm3.7} is an immediate consequence of the following theorem.
\end{proof}

\noindent {\bf Theorem C.} (Gy\H{o}ry \cite{gy2008}) {\sl Let $A$ be an ordered $k$-term subset in $\Z_S$. If
$$
k>3\cdot 2^{16(|S|+2)}
$$
then $\overline{\G_S(A)}$ has at most two components, and one of them is of order at most $1$.}

\begin{proof} This is a special case of Theorem 2.3 of \cite{gy2008}.
\end{proof}

\section{Proofs of the results stated in Section \ref{s3}}
\label{s-p3}

\begin{proof}[Proof of Theorem \ref{thm3.1}] If $G$ has no edges, then it is infinitely representable with all $S$. If $G$ is not connected, but is representable with some $S$, then following the proof of Theorem \ref{thmuj} one can easily see that $G$ is infinitely representable with $S$. So assume that $G$ is connected, but not doubly connected. Then $G$ has a bridge, i.e. an edge $e$ such that $G-\{e\}$ is the union of two components, say $G_1$ and $G_2$. We may further suppose that $|G_2|\geq 2$.

Assume that $G$ is representable for some $S$. Let $A$ and $B$ be sets of $S$-integers corresponding to the vertices of the components $G_1$ and $G_2$, respectively. Write $A=\{a_0,a_1,\dots,a_n\}$ and $B=\{b_0,b_1,\dots,b_k\}$. Without loss of generality we may assume that $e$ connects the vertices corresponding to $a_0$ and $b_0$ in $G$.

We show that then $G$ is infinitely representable with $S$. For this consider the sets $A+a'$ where $a'$ is such that
$a'-a_0+b_0\in \Z_S^*$. Write $u:=a'-a_0+b_0$ and $w:=a'-a_i+b_j$ for some arbitrary $(i,j) \not= (0,0)$. Then
\begin{equation}
\label{eq111}
u-w=a_i-b_j-a_0+b_0.
\end{equation}
If the right hand side is zero, then $a_0-b_0=a_i-b_j$ $((i,j)\neq (0,0))$ would be valid. Since $e$ corresponding to $a_0-b_0$ is an edge of $G$, hence $a_0-b_0\in \Z_S^*$, we also would have $a_i-b_j\in \Z_S^*$ so that $a_i$ and $b_j$ would also be connected by an edge, contradicting the assumption that $a_0-b_0$ is a bridge between $A$ and $B$. Thus the right hand side of \eqref{eq111} is nonzero. Since \eqref{eq111} has only finitely many solutions in $u,w\in \Z_S^*$, there exist infinitely many $u\in \Z_S^*$ such that the corresponding $w$ is not in $\Z_S^*$, thus $a'-a_i+b_j \notin \Z_S^*$. Since it is true for all $(i,j) \neq(0,0)$, we obtain that for infinitely many $u\in \Z_S^*$ we have $a'-a_i+b_j \notin \Z_S^*$, for all $(i,j)\neq (0,0)$. Now choosing $a'$ accordingly, $(A+a') \cup B$ has the same induced graph as $A \cup B$.
\end{proof}

\begin{proof}[Proof of Theorem \ref{thm2}] We may assume that $G$ is connected, otherwise by Theorem \ref{thmuj} we may apply our argument to the components of $G$. Further, if $G$ is a tree then by Theorem \ref{thm3} it can be represented by any $S$ and by Theorem \ref{thm3.1} we are done. Thus, in particular, we may suppose that $G$ contains a cycle.

Let $p\in{\mathbb Z}$ be a prime larger than $n$, the number of edges of $G$, and suppose that $G$ is representable with $S_0=\{p\}$ and let $A_0\subseteq\Z_S$ be a finite set such that $G$ is isomorphic to the induced graph $\G_{S_0}(A_0)$. Note that, as before, we may assume that $A_0\subseteq \Z$.  Without loss of generality we may assume that $0\in A_0$. Label the edges of $G$ by the corresponding $S_0$-units, and write
$$
E_0=\{\pm u_1,\dots,\pm u_k\}
$$
for the set of occurring $S_0$-units. We assume here that $u_i>0$ $(i=1,\dots,k)$. Note that it may happen that some $u_i$ or $-u_i$ labels more edges.

Suppose that $k=1$. Considering any cycle of $G$ we see that both $u_1$ and $-u_1$ must occur as labels of some edges. However, then there must be two consecutive edges in that cycle with labels $u_1$ and $-u_1$ (or vice versa), which yields a contradiction. So we conclude that $k\geq 2$.

Take now any finite set $S$ of primes, and choose arbitrary $S$-units $w_1,\dots,w_k$ such that $2n|w_i|<|w_{i+1}|$ $(i=1,\dots,k-1)$. Recall that $n$ stands for the number of edges of $G$. Replace the labels $u_i$ by $w_i$ and $-u_i$ by $-w_i$ for all $i=1,\dots,k$ over the edges of $G$, and write $0$ for that vertex of $G$ which corresponds to $0\in A_0$ in the isomorphism $G\sim \G_{S_0}(A_0)$. Starting from this vertex $0$, attach values to the vertices of $G$ in the following way. Take an arbitrary walk from $0$ to a vertex $v$, and add the $S$-units over the labels on the path, to get the value of $v$. We show that the values of the vertices are well-defined. Let $v$ be any vertex of $G$, and let $e_1,\dots,e_t$ and and $e_1^*,\dots,e_\ell^*$ be two sequences of edges yielding walks from $0$ to $v$. If the values of $v$ obtained by using these walks are different, then the sum over the (appropriately directed) edges of the cycle $e_1,\dots,e_t,-e_\ell^*,\dots,-e_1^*$ does not vanish. This yields that for some $i\in\{1,\dots,k\}$ there are more edges with label $w_i$ than with $-w_i$ (or vice versa) in the cycle. However, then this is valid in the original labeling for $u_i$ and $-u_i$. This by $p>2n$ yields a contradiction.

Now we show that the values of the vertices are distinct. Suppose to the contrary that two such values coincide. This gives rise to an equality of the form
$$
c_{i_1}w_{i_1}+\dots+c_{i_t}w_{i_t}-
(c_{j_1}w_{j_1}+\dots+c_{j_\ell}w_{j_\ell})=0,
$$
where $w_{i_1},\dots,w_{i_t}$ are the edge labels along a path from $0$ to the one vertex, and $w_{j_1},\dots,w_{j_\ell}$ are the edge labels along a path from $0$ to the other vertex. Observe that $t,l \leq n$, and the coefficients are from $\{\pm1\}$. Since $2n|w_i|<|w_{i+1}|$ for all $i$, this is possible if and only if in the above equation after cancelations the coefficients of the $w_i$'s are all zero. However, the same identity holds for the $u_i$'s. This is a contradiction, since then we would have coinciding vertices in $A_0$.

Subsequently, we prove that the induced graph $\G_S(A)$ is isomorphic to $G$. Here $A$ is the set of $S$-integers defined in the natural way, by attaching to a vertex $v$ the sum of the values $w_i$ corresponding to the edges of a path from $0$ to $v$. Since it is obvious that $G\subseteq \G_S(A)$, we only need to check that $\G_S(A)$ does not contain more edges than $G$ does. This follows from the above proved fact that $\sum_{j=1}^t c_{i_j} w_{i_j} =0$ with $|c_{i_j}| < n$ implies $c_{i_1} = \dots = c_{i_t} = 0$ and therefore $\sum_{j=1}^t c_{i_j}u_{i_j} = 0$.
Indeed, this shows that if two vertices would be connected by an edge in $\G_S(A)$, then they also would be connected by an edge in $\G_{S_0}(A_0)$ as well, hence also in $G$.

Finally, the infinitude of representations follows from $k\geq 2$, as we have infinitely many choices for $w_k$.
\end{proof}

\begin{proof}[Proof of Corollary \ref{cor4.1}]
Let $k$ be the number of vertices of $G$. Let $p$ be a prime number $>k$. Put $S'= \{p\}$.
According to Theorem \ref{thm2}, $G$ is not representable with $S'$.
\end{proof}

\begin{proof}[Proof of Corollary \ref{cor4.2}]
Straightforward consequence of Theorem \ref{thm2}.
\end{proof}

\section{Proofs of the results stated in Section \ref{s-tri}}
\label{s-p4}

The following theorem is the main ingredient of the proof of Theorem \ref{thm5.4}. It was established in terms of the complements of the graphs $\G_S(A)$ which formulation is more useful for certain applications.

\vskip.4truecm

\noindent
{\bf Theorem D.}  (Gy\H{o}ry \cite{gy2008}) {\sl
Let $k\geq 3$ be an integer, and fix $S$. Then for all but at most
$$
\left((k+1)^42^{16(|S|+2)}\right)^{k-2}
$$
$S$-equivalence classes of ordered $k$-term subsets $A$ from $\Z_S$, one of the following cases holds:

\vskip.2truecm

\noindent
i) $\overline{\G_S(A)}$ is connected and at least one of $\G_S(A)$ and $\G_S(A)^\triangle$ is not connected;

\vskip.2truecm

\noindent
ii) $\overline{\G_S(A)}$ has exactly two components, $\overline{\G_1}$, and $\overline{\G_2}$, say, such that $|\overline{\G_1}|=1$, and $\G_2$ is not connected;

\vskip.2truecm

\noindent
iii) $\overline{\G_S(A)}$ has exactly two components of orders $\geq 2$.
}

\vskip.3truecm

\begin{proof}
This is an immediate consequence of a special case of Theorem 2.2 of \cite{gy2008}.
\end{proof}

\vskip.3truecm

\noindent {\bf Remark 1.} For earlier versions of Theorem D, we refer to \cite{gy1980,gy1982,gy1990,gy1992}. A less precise version in \cite{gy1980} is effective.

\vskip.3truecm

\noindent {\bf Remark 2.} We note that in Theorem D one could consider more generally so-called polygon hypergraphs $\G_S(A)^\circ$ in place of $\G_S(A)^\triangle$, where however, only those cycles $a_{i_1},\dots,a_{i_\ell}$ $(\ell\geq 3)$ are taken into consideration in $\G_S(A)^\circ$ for which
$$
\sum\limits_{j\in J} (a_{i_j}-a_{i_{j+1}})\neq 0\ \text{for each non-empty subset $J$ of $\{1,\dots,\ell\}$},
$$
see \cite{gy1990,gy2008}. Moreover, in this case the situation iii) cannot occur if $k\neq 4$ and one can also obtain an explicit upper bound for the number of exceptional $S$-equivalence classes. However, for abstract graphs this ``non-degeneracy" concept cannot be adapted. Hence we shall work here with $\triangle$-connectedness only.

The innovation in Theorem \ref{thm5.4} concerns part iii). For $k\geq 5$, the following lemma provides an upper bound for the number of cases in Theorem D iii).

\begin{lemma}
\label{lem5.1}
Let $k\geq 5$ be an integer, and let $S$ be fixed. There are at most
$$
\left(k\cdot 5^{648(3|S|+4)}\right)^{k-1}
$$
$S$-equivalence classes of ordered $k$-term subsets $A$ in $\Z_S$ for which $\overline{\G_S(A)}$ consists of two components, of which one has order $\geq 3$ and the other has order $\geq 2$.
\end{lemma}

In the proof of Lemma \ref{lem5.1} we use the following result.

\begin{lemma}
\label{lem5.3}
Apart from an $S$-unit factor, there are at most
$$
24^{324(3|S|+4)}
$$
elements $c\in \Q^*$ such that
\begin{equation}
\label{neweq9}
x+y=c\ \ \ \text{in}\ x,y\in \Z_S^*
\end{equation}
has more than two solutions.
\end{lemma}

\noindent For the finiteness of the number of elements $c\in Q^*$ in Lemma \ref{lem5.3}, see Evertse, Gy\H{o}ry, Stewart and Tijdeman \cite{egst}.

\begin{proof}[Proof of Lemma \ref{lem5.3}]
Assume that there are at least three solutions. Then without loss of generality we may assume that $(x,y)$ and $(x',y')$ are solutions of \eqref{neweq9} such that $(x',y')\neq (x,y),(y,x)$. Since
$$
x+y=x'+y',
$$
it follows that $(x/y',y/y',-x'/y')$ is a non-degenerate solution of the equation
$$
x_1+x_2+x_3=1\ \ \ \text{in}\ x_1,x_2,x_3\in \Z_S^*.
$$
Then Theorem B implies that there are at most $N_3\leq 24^{324(3|S|+4)}$ possibilities for $(x/y',y/y')$ and hence for $c/y'$. This proves the assertion.
\end{proof}

\begin{proof}[Proof of Lemma \ref{lem5.1}]
Let $A=\{a_1,\dots,a_k\}$ be an ordered $k$-term subset from $\Z_S$ with $k\geq 5$ such that $\overline{\G_S(A)}$ has two components $\overline{\G_S(A_m)}$ and $\overline{\G_S(A_n)}$, where
$$
A_m=\{a_1,\dots,a_m\},\ A_n=\{a_{m+1},\dots,a_k\},\ m+n=k
$$
and $m\geq 3$, $n\geq 2$. Then
$$
a_i-a_j\in \Z_S^*\ \ \ \text{for}\ 1\leq i\leq m,\ m+1\leq j\leq k.
$$

Let
$$
C_1=24^{324(3|S|+4)},\ C_2=3\cdot 7^{2|S|+3}.
$$
We have
$$
a_{m+1}-a_{m+2}=(a_{m+1}-a_i)+(a_i-a_{m+2})\ \ \ \text{for}\ i=1,\dots,m.
$$
But $a_{m+1}-a_i,a_i-a_{m+2}\in \Z_S^*$ for each $i$ with $1\leq i\leq m$. Since by assumption $m\geq 3$, Lemma \ref{lem5.3} implies that
$$
a_{m+1}-a_{m+2}=u_{m+1,m+2}a_{m+1,m+2}
$$
where $a_{m+1,m+2}$ may take at most $C_1$ values and $u_{m+1,m+2}\in \Z_S^*$.

Put $a_i'=a_i/u_{m+1,m+2}$ for $i=1,\dots,k$, $A'=\{a_1',\dots,a_k'\}$, $A_m'=\{a_1',\dots,a_m'\}$, $A_n'=\{a_{m+1}',\dots,a_k'\}$ and fix the value of $a_{m+1,m+2}=a_{m+1}'-a_{m+2}'$. Then
$$
a_{m+1,m+2}=(a_{m+1}'-a_i') +(a_i'-a_{m+2}'),\ \ \ i=1,\dots,m,
$$
where $a_{m+1}'-a_i',a_i'-a_{m+2}'$ are $S$-units. By Theorem A there are at most $C_2$ such pairs of $S$-units. Taking differences, we infer that $a_i'-a_1'$ may take at most $C_2^2$ values for $i=2,\dots,m$. Then the number of possible tuples $a_2'-a_1',\dots,a_m'-a_1'$ is at most $(C_1C_2^2)^{m-1}$. But for fixed $a_2'-a_1'$ and for $m<\ell\leq k$ we have
$$
a_2'-a_1'=(a_2'-a_\ell')+(a_\ell'-a_1')
$$
where $a_2'-a_\ell'$, $a_\ell'-a_1'$ are $S$-units and they may take at most $C_2$ values.

Putting $A_0=A'-a_1'$, the number of possible ordered $k$-term subsets $A_0$ in $\Z_S$ is at most $(C_1\cdot C_2^2)^{k-1}$. Further, $A=uA_0+b$ with $u=u_{m+1,m+2}\in \Z_S^*$ and $b=a_1'u_{m+1,m+2}\in \Z_S$.

Finally, for a fixed ordering of the elements $a_1,\dots,a_k$, the integers $m,n$ can be chosen in at most $k-4$ ways. Further, the number of possible orderings of elements of $A$ is at most $k!$. Hence the total number of ordered $k$-term subsets $A$ does not exceed
$$
(k-4)k!(C_1C_2^2)^{k-1},
$$
whence, after some computation, the assertion follows.
\end{proof}

\begin{proof}[Proof of Theorem \ref{thm5.4}] Combine Theorem D and Lemma \ref{lem5.1}.
\end{proof}

\begin{proof}[Proof of Theorem \ref{thm3.8}] Let $G$ be a graph of order $\geq 3$ and $S$ a finite set of prime numbers. Suppose that $G$ is isomorphic to $G_S(A)$ for more than
$$
\left(k\cdot 5^{162(3|S|+4)}\right)^{4(k-1)}
$$
$S$-equivalence classes of ordered subsets $A$ from $\Z_S$. Theorem \ref{thm5.4} implies that for these subsets $A$, i), ii) or iii) holds. Observe that in cases ii) and iii) $\G_S(A)^\triangle$ is not connected. Because of the isomorphy of $\G_S(A)$ and $G$, the assertion immediately follows.
\end{proof}

\begin{proof}[Proof of Theorem \ref{thm3.3}] This is an immediate consequence of Theorem \ref{thm3.8}.
\end{proof}

\begin{proof}[Proof of Theorem \ref{thm3.4}]
Let $G$ be a graph of order $\geq 3$. Suppose that $G$ is representable with some $S$ and that $G$ and $\mH(G)$ are connected. If $G^\triangle$ is connected then the assertion follows from Theorem \ref{thm3.3}. Consider the case when $G^\triangle$ is not connected. By Theorem \ref{thm3.3} each $\triangle$-connected component of $G^\triangle$ is finitely representable. Further, we claim that if two such components are connected in $\mH(G)$ then the subgraph of $G$ spanned by these components is also finitely representable.

Indeed, let $\G_S(A)$ be a graph isomorphic to $G$ for some subset $A$ of $\Z_S$, and let $\G_S(B)$, $\G_S(B')$ be the induced subgraphs of $\G_S(A)$, isomorphic to the respective subgraphs of $G$ spanned by the two components under consideration. Then it follows that
$$
b-c =u r_{b,c}\ \ \ \text{and}\ \ \ b'-c'=w r'_{b',c'}
$$
for each distinct $b,c$ from $B$ and $b',c'$ from $B'$, where $u,w$ are $S$-units and $r_{b,c}$, $r'_{b',c'}$ can take only finitely many values from $\Z_S$. But by assumption $B$ and $B'$ have two common vertices, which implies that $w=ut$ for some $t\in \Z_S$ which may take only finitely many values. For each $b\in B$ and $b'\in B'$ we have
$$
b-b'=(b-c)+(c-b')
$$
where $c$ is a common vertex in $B$ and $B'$. This means that up to the factor $u$, $b-b'$ may take only finitely many values from $\Z_S$, whence our claim follows.

Finally, we can proceed by adding component after component in the same way, and the assertion follows by induction.
\end{proof}

\begin{proof}[Proof of Theorem \ref{thm3.6}] The theorem directly follows from Theorem C.
\end{proof}

\section{Proofs of the results stated in Section \ref{s-i}}
\label{s-p5}

\begin{proof}[Proof of Theorem \ref{thmcub}]

Suppose $G$ is cubical. Choose an integer $n$ such that $G$ can be embedded in $Q_n$.
Then the vertices of $G$ can be denoted by vectors $(a_1, \dots, a_n) \in \{0,1\}^n$ and two vertices are connected if and only if their difference is a unit vector $ \pm \vec{e}_i$ for some $i$. By Theorem \ref{thm2} it suffices to prove the statement for $S=\{p\}$, where $p$ is a prime larger than the number of edges of $G$. Assign to the vertex $(a_1, \dots, a_n)$ the value $\sum_{i=1}^n a_i p^{i}$.
If two vertices are adjacent in $Q_n$, then they differ by a unit vector. Hence their values differ by a power of $p$ which is in $U$ and therefore they are connected in $G$.
If two vertices are not adjacent in $Q_n$, then they differ by a vector $(b_1, \dots, b_n)$ with $b_i \in \{-1,0,1\}$ and at least two entries nonzero. Let $i_0$ be the smallest index with $b_{i_0}\not= 0$. Then their values differ by $p^{i_0}\sum_{i=i_0}^n b_i p^{i-i_0}$. Since $\sum_{i=i_0}^n b_i p^{i-i_0} $ is nonconstant and coprime to $p$, we have that $\sum_{i=1}^n b_i p^{i}$ is not in $U$. Thus $G$ is representable for $S=\{p\}$.

\vskip 2mm

Suppose $G$ is representable with all $S$. Without loss of generality we may assume that $G$ is connected. Otherwise we apply the argument below to each component of $G$.
Let $k$ be the number of edges of $G$. Let $S = \{p\}$ where $p$ is a prime greater than $k$.
Since $G$ is represented with $S$, we can adjoin the value 0 to one vertex of $G$ and then values to all other vertices are induced by adding the labels of the edges along a path from the origin to that vertex.
As we consider a representation of $G$ with $S$, the difference between the values of two vertices is a power of $p$ if and only if the vertices are adjacent.
By the choice of $p$, for every positive integer $m$ every cycle in $G$ contains as many edges with value $p^m$ as edges with value $-p^m$.
Suppose that $M= \{\pm p^{m_1},\pm p^{m_2},\dots,\pm p^{m_r} \}$ is the set of labels of the edges which occur.
Then the values of the vertices are of the form $a_1 p^{m_1} + a_2 p^{m_2} + \dots + a_r p^{m_r}$ where $a_1, \dots, a_r$ are integers with $|a_i| < p$ for $i=1, \dots, r$.
It follows that $G$ can be embedded in $\mathbb{Z}^r$ by mapping the vertex with value $a_1 p^{m_1} + a_2 p^{m_2} + \dots +a_r p^{m_r}$ to $(a_1, \dots, a_r)$.
All the vertices are in the hypercube $[-p+1,p-1]^n$.
Subsequently we introduce unit vectors $\vec{e}_{i,j}$ for $i=1, \dots, r$ and $j= -p, \dots, p-1$.
The edge connecting a vertex $(a_1, \dots, a_r)$ with a vertex $(a_1, \dots ,a_r) + \vec{e}_i$ gets the new value $\vec{e}_{i,a_i}$.
By doing so a path in the hypercube $\mathbb{Z}_r$ is mapped to a path in the hypercube $Q_{2pr}$.
It is still true that two vertices of $G$ are adjacent if and only if their difference in $Q_{2pr}$ is a unit vector.
Hence $G$ is cubical.
\end{proof}

\noindent {\bf Remark.} Note that in the above proof we have constructed a $c$-valuation in the sense of Havel and Moravek \cite{hm}.

\section{Acknowledgements}

We are grateful to Professor L. Lov\'asz and Dr. A. Mar\'oti for their useful remarks and for calling our attention to some references. Further, we thank the referees for their useful comments.

\end{document}